\documentclass[11pt,a4paper,oneside]{amsart}

\usepackage{amsfonts, amsmath, amssymb, amsthm, amscd}
\usepackage{hyperref,paralist}
\usepackage[arrow,curve,matrix,tips,2cell]{xy}
  \SelectTips{eu}{10} \UseTips
  \UseAllTwocells
\hypersetup{
  colorlinks   = true, 
  urlcolor     = blue, 
  linkcolor    = blue, 
  citecolor   = red 
}
\usepackage{anysize}

\newtheorem{theorem}{Theorem}[section]
\newtheorem{lemma}[theorem]{Lemma}

\theoremstyle{definition}
\newtheorem{definition}[theorem]{Definition}

\theoremstyle{remark}
\newtheorem{remark}[theorem]{Remark}

\numberwithin{equation}{section}

\newcommand\Sym{\mathfrak S}
\newcommand{\R}{\mathbb{R}}
\newcommand{\Z}{\mathbb{Z}}
\newcommand{\F}{\mathbb{F}}
\newcommand{\B}{\mathrm{B}}
\newcommand{\E}{\mathrm{E}}
\newcommand{\BO}{\mathrm{BO}}
\newcommand{\EO}{\mathrm{EO}}
\newcommand{\OO}{\mathrm{O}}
\newcommand{\cp}{\mathfrak c}
\newcommand\oo{\mathfrak o}

\DeclareMathOperator{\depth}{depth}
\DeclareMathOperator{\conv}{conv}
\DeclareMathOperator{\relint}{relint}

\DeclareMathOperator{\vol}{vol}

\begin{document}

\title{A center transversal theorem for an improved Rado depth}

\author{Pavle V. M. Blagojevi\'{c}}
\thanks{P. B. is supported by DFG via the Collaborative Research Center TRR~109 ``Discretization in Geometry and Dynamics'', and the grant ON 174008 of the Serbian Ministry of Education and Science.}
\address{Inst. Math., FU Berlin, Arnimallee 2, 14195 Berlin, Germany\hfill\break%
\mbox{\hspace{4mm}}Mat. Institut SANU, Knez Mihailova 36, 11000 Beograd, Serbia}
\email{blagojevic@math.fu-berlin.de}

\author{Roman~Karasev}

\thanks{R. K. is supported by the Federal professorship program grant 1.456.2016/1.4, the Russian Science Foundation grant 18-11-00073, and the Russian Foundation for Basic Research grant 18-01-00036}

\address{Moscow Institute of Physics and Technology, Institutskiy per. 9, Dolgoprudny, Russia 141700\hfill\break
\mbox{\hspace{4mm}}Institute for Information Transmission Problems RAS, Bolshoy Karetny per. 19, Moscow, \hfill\break 
\mbox{\hspace{4mm}}Russia 127994}

\email{r\_n\_karasev@mail.ru}
\urladdr{http://www.rkarasev.ru/en/}

\author{Alexander Magazinov}

\address{School of Mathematics, Tel Aviv University, 69978 Tel Aviv, Israel}

\email{magazinov@post.tau.ac.il}

\thanks{A. M. is supported by ERC Advanced Research Grant no. 305629 (DIMENSION)}

\begin{abstract}
A celebrated result of Dol'nikov, and of \v{Z}ivaljevi\'c and Vre\'cica, asserts that for every collection of $m$ measures $\mu_1,\dots,\mu_m$  on the Euclidean space $\mathbb R^{n + m - 1}$ there exists a projection onto an $n$-dimensional vector subspace $\Gamma$ with a point in it at depth at least $\tfrac{1}{n + 1}$ with respect to each associated $n$-dimensional marginal measure $\Gamma_*\mu_1,\dots,\Gamma_*\mu_m$.

In this paper we consider a natural extension of this result and ask for a minimal dimension of a Euclidean space in which one can require that for any collection of $m$ measures there exists a vector subspace $\Gamma$ with a point in it at depth slightly greater than $\tfrac{1}{n + 1}$ with respect to each $n$-dimensional marginal measure.
In particular, we prove that if the required depth is $\tfrac{1}{n + 1} + \tfrac{1}{3(n + 1)^3}$ then the increase in the dimension of the ambient space is a linear function in both $m$ and $n$.
\end{abstract}

\maketitle

\section{Introduction}

We start by introducing a notion of point depth with respect to a given measure in a Euclidean space.
For an analogous notion in a discrete setting, in case of point sets, consult for example ~\cite[Sec.\,1.1]{BMN10}.

An oriented affine hyperplane $H_{v,a}=\{x\in\R^d : \langle x,v\rangle =a\}$ is given by a unit vector $v\in S(\R^d)$ and a constant $a\in\R$.
It determines two closed half-spaces, which we denote by
\[
H^{0}_{v,a}=\{x\in \R^d :\langle x,v\rangle \geq a\},\qquad H^{1}_{v,a}=\{x\in\R^d : \langle x,v\rangle \leq a\}.
\]
The notation $\langle x,v \rangle$ refers to the standard Euclidean scalar product in $\R^d$.
If the constant $a$ is zero then the hyperplane $H_{v,a}$ is an oriented linear hyperplane.

\begin{definition}
\label{def : depth}
Let $N \geq 1$ be an integer, let $x \in \R^N$ be a point, and let $\mu$ be a Borel probability measure on the same space $\R^N$.
The {\bfseries depth of the point} $x$ with respect to the measure $\mu$ is:
\[
\depth_{\mu}(x) := \inf \{\mu(H_{v,a}^0) : H_{v,a} \text{ is an oriented affine hyperplane with }x \in H_{v,a}^0\}.
\]
\end{definition}

In order to distinguish just introduced notion of depth from other notions of depth we recall that the depth we consider is also called the half-space depth, or the Tukey depth~\cite{Tuk75}.
If the measure is clear from the context we omit it from the notation and simply write $\depth(x)$.
Throughout the paper by ``a measure'' we always mean ``a Borel probability measure''.

\medskip

Next we introduce the notion of a marginal, or a projection, of a measure on the Euclidean space $\R^N$ with respect to an affine subspace of $\R^N$.

\begin{definition}
Let $\Gamma$ be an affine subspace of the Euclidean space $\R^N$, and let $\mu$ be a measure on $\R^N$.
For every Borel set $X \subseteq \Gamma$ define
\[
(\Gamma_*\mu)(X): = \mu(\pi_{\Gamma}^{-1}(X))
\]
where $\pi_{\Gamma}\colon\R^N\longrightarrow\Gamma$ denotes the orthogonal projection of $\R^N$ onto $\Gamma$.
Clearly, $\Gamma_*\mu$ is a probability measure on $\Gamma$, and is called a {\bf marginal}, or a {\bf projection}, of the measure $\mu$ with respect to $\Gamma$.
\end{definition}

\medskip
The motivation for the study of the point depth of measures comes from the following result of Dol'nikov~\cite{Dol87} \cite{Dol92}, and \v{Z}ivaljevi\'c and Vre\'cica~\cite{ZV90}, and many of its applications.

\begin{theorem}[Center transversal theorem]
Let $m\geq 1$, $n\geq 1$ and $N \geq 1$ be integers with $N \geq m + n - 1$.
For every collection of $m$ measures $\mu_1, \dots, \mu_m$ on $\R^N$ there exists an $n$-dimensional linear subspace $\Gamma$ and a point $x \in \Gamma$ such that for every $1\leq i\leq m$:
\[\depth_{\Gamma_*\mu_i}(x) \geq \frac{1}{n + 1}.\]
\end{theorem}

\noindent
This theorem is a direct extension of the classical Rado theorem, also known as the centerpoint theorem~\cite{Ra46}, which states that every measure $\mu$ on $\R^N$ has a point $x$ with $\depth_{\mu}(x) \geq \tfrac{1}{N + 1}$.
The Rado theorem is a particular case of the center transversal theorem when $m = 1$.

\medskip
A new way to extend the result of Rado was proposed in~\cite{BMN10}.
Namely, given a measure $\mu$ in $\mathbb R^N$, one may wish to find an $n$-linear subspace $\Gamma$ and a point $x \in \Gamma$ such that $\depth_{\mu_{\Gamma}}(x)$ is as large as possible.
The Rado theorem immediately implies that $\depth_{\Gamma_*\mu}(x) \geq \tfrac{1}{N + 1}$ is always possible, and one may ask whether this estimate (called the Rado bound) can be improved.
Since the Rado theorem is optimal with respect to the dimension of the ambient space such an improvement can only come from an increase in the dimension and careful choice of an affine subsace $\Gamma$.
Recently, it was demonstrated in~\cite[Thm.\,2]{MP16} how the Rado bound can indeed be surpassed.
Namely, the following so called Centerline theorem holds.

\begin{theorem}[Centerline theorem]
\label{Theorem:Centerline}
Let $N \geq 3$ be an integer, and set $n: = N - 1$.
For every measure $\mu$ on $\R^N$, there exists an $n$-dimensional linear subspace $\Gamma$, in this case a hyperplane, and a point $x \in \Gamma$ such that
$$
\depth_{\Gamma_*\mu}(x) \geq \frac{1}{n + 1} + \frac{1}{3(n + 1)^3}=\frac{1}{N} + \frac{1}{3N^3}.
$$
\end{theorem}

\begin{remark}
The theorem is given such a name because the line $\pi_{\Gamma}^{-1}(x)$ is a natural candidate to play a role of a centerline.
\end{remark}

At this point it is important to observe that the depth bound in the center transversal theorem coincides with the Rado bound.
In this paper we combine the features of the center transversal theorem (several measures are considered at once) and of the centerline theorem (surpassing the Rado bound at a cost of increasing the dimension of the ambient space).
Motivated by the previous attempt of the third author (see the first version of this paper~\cite{Magazionov2016}), using advanced methods of algebraic topology, we prove in Section \ref{sec : proof} the following center transversal theorem with an improved Rado depth.

\begin{theorem}[Center transversal theorem with an improved Rado depth]
\label{Th : Main}
Let $m \geq 1$, $n \geq 2$ be integers, and let
\begin{compactitem}
	\item $N \ge 2m+n-1$ if $n+1$ is not a power of $2$, and
	\item $N \ge 3m+n-1$ if $n+1$ is a power of $2$.
\end{compactitem}
For every collection of $m$ measures $\mu_1, \dots, \mu_m$ on $\R^N$, there exists an $n$-dimensional linear subspace $\Gamma$ and a point $x \in \Gamma$ such that for every $1\leq i\leq m$:
$$
\depth_{\Gamma_*\mu_i}(x) \geq \frac{1}{n + 1} + \frac{1}{3(n + 1)^3}.
$$
\end{theorem}

There are two alternative approaches that might be employed to obtain a similar result.
One approach is related the argument given by Bukh, Matou\v{s}ek, and  Nivasch in \cite{BMN10}, and the other one relies on the work of Klartag~\cite{Kl10}.
Both of them yield an estimate of the form $N \geq m f(n)$. Using the topological approach of~\cite{BMN10} one arrives to $f(n)$ growing exponentially in $n$. Using the method of ``almost orthogonal
decomposition'' of a measure as in~\cite{Kl10} one can obtain $f(n) \sim n^p$, where $p$ is a sufficiently large constant (at least, $p > 2$).

Natural questions arise concerning the optimality of Theorem~\ref{Th : Main}. Are the bounds of Theorem~\ref{Th : Main} for $N$ optimal? If they are not optimal, and there is a bound $N > n + cm + c'$,
is it true that $c > 1$? (Note that the original Center transversal theorem features a bound of that form with $c = 1$.) The author do not know the answers to these questions.

\medskip
\noindent
{\bf Acknowledgements.}
The authors are grateful to Vladimir Dol'nikov and Gaiane Panina for suggesting the several-measure setup, and to Bo'az Klartag for pointing out some of the needed constructions.
Furthermore we want to express our gratitude to the referee for excellent observations and many useful comments.

\section{Structures assigned to a measure}

All the results in the sections to come will be first established for {\bf nice measures}.
These are measures with continuous density functions whose support is compact and connected.
Then we prove Theorem~\ref{Th : Main} for nice measures as well.
After that the general case of Theorem~\ref{Th : Main} follows from the ``nice'' case by a classical approximation argument limiting the space of feasible solutions to a compact set.

\begin{definition}
Let $n \geq 1$ be an integer and let $\mu$ be a measure on the Euclidean space $\R^n$.
The {\bf depth of the measure} $\mu$ is:
\[
\depth(\mu):= \sup\limits_{x \in \mathbb R^n} \depth_{\mu}(x).
\]
\end{definition}

\begin{definition}
For a given nice measure $\mu$ on $\R^n$, we specify the point $\cp(\mu)$ associated to the measure $\mu$ as follows:
\begin{compactitem}
\item If $\depth(\mu) \le \frac{1}{n + 1} + \frac{1}{3(n + 1)^3}$ then define  $\cp(\mu)\in\R^n$ to be the unique point of maximal depth with respect to $\mu$, that is $\depth (\mu)=\depth_{\mu}(\cp(\mu))$.
	(The existence and uniqueness of such a choice was established in~\cite[Lemma\,3.1]{MP16}.)

\item If $\depth(\mu) > \frac{1}{n + 1} + \frac{1}{3(n + 1)^3}$ consider the set
      \[
      K(\mu) = \Big\{ x \in \R^n : \depth_{\mu}(x) \geq \frac{1}{n + 1} + \frac{1}{3(n + 1)^3} \Big\}.
      \]
      It is a convex $n$-dimensional body as a complement of a union of open half-spaces.
      Now we define $\cp(\mu)$ to be the barycenter of $K(\mu)$.
\end{compactitem}
\end{definition}

The point $\cp(\mu)\in\R^n$ depends continuously on the measure $\mu$.
In order to verify this at the measure $\mu_0$ we consider three cases depending on the depth of the measure $\mu_0$:
\begin{compactenum}[\rm ~~(1)]
  \item If $\depth(\mu_0) < \frac{1}{n + 1} + \frac{1}{3(n + 1)^3}$,  then continuity is proved in~\cite[Lemma~4]{MP16}.

  \item If $\depth(\mu_0) > \frac{1}{n + 1} + \frac{1}{3(n + 1)^3}$, then the following pointwise convergence holds:
        \[
        \frac{\mathbf 1_{K(\mu) \cup \partial K(\mu_0)}}{\vol\, K(\mu)} \longrightarrow \frac{\mathbf 1_{K(\mu_0)}}{\vol\, K(\mu_0)}, \qquad\qquad \text{for $\mu \longrightarrow \mu_0$.}
        \]
        The Bounded convergence theorem for integrals \cite[Ex.\,1.5.18]{Tao2011} of the above functions implies the continuity of $\cp(\mu)$.
  \item If $\depth(\mu_0) = \frac{1}{n + 1} + \frac{1}{3(n + 1)^3}$ then we can repeat the argument of~\cite[Lemma~4]{MP16} to show that the point $\cp(\mu)$ cannot escape from any fixed neighborhood of $\cp(\mu_0)$ provided that the measure $\mu$ is sufficiently close to the measure $\mu_0$.
\end{compactenum}

\medskip
Furthermore, it is shown in~\cite[Lemma~5]{MP16} that, if a nice measure $\mu$ on $\R^n$ has {\bf insufficient depth}, that is,
\[
\depth(\mu) < \frac{1}{n + 1} + \frac{1}{3(n + 1)^3}
\quad \Longleftrightarrow\quad
\depth_{\mu} \cp(\mu) < \frac{1}{n + 1} + \frac{1}{3(n+ 1)^3},
\]
 then it is possible to construct a set of points $\{v_0(\mu), v_1(\mu), \dots, v_n(\mu)\}$ in $\R^n$, that depends continuously on the measure $\mu$, satisfying
\[
\mathbf 0 \in \relint \big(\conv \{ v_0(\mu), v_1(\mu), \dots, v_n(\mu) \}\big) .
\]

The intuition behind this phenomenon might become more evident by considering a ``typical'' measure of insufficient depth. 
Let points $e_0, e_1, \dots, e_n \in \R^n$ satisfy $\mathbf 0 \in \relint \big(\conv \{ e_0, e_1, \dots, e_n \}\big)$, and let
\[
\mu = \tfrac{1}{d + 1}(\nu_0 + \nu_1 + \dots + \nu_n),
\]
where $\nu_i$ is a nice measure sharply concentrated around $e_i$. (We also require $o(\nu) = \mathbf 0$, but this can also be
settled by the particular choice of $\nu_i$.) It is not hard to check that $\depth_{\nu}(0)$ is close to $\tfrac{1}{d + 1}$,
so $\mu$  is indeed a measure of insufficient depth. If we were restricted only to this type of measures, then it would
have been natural to put $v_i = e_i$.

Given a set $\{v_0(\mu), v_1(\mu), \dots, v_n(\mu)\}$, there is a unique, up to multiplying by a common positive factor, positive dependence
\[
\lambda_0 v_0(\mu) + \lambda_1 v_1(\mu) + \dots + \lambda_n v_n(\mu) = 0, \qquad\qquad \lambda_i > 0. 
\]
The additional condition that $\vol \big( \conv \{ \lambda_i v_i : 0\leq i\leq n \}\big) = 1$ determines a unique simplex $\Sigma(\mu) = \conv \{ \lambda_i v_i : 0\leq i\leq n \}$,
which depends continuously on $\mu$.

Let $A_{\mu}$ be a linear map such that $\Sigma(\mu) = A_{\mu}\Delta_n$, where $\Delta_n$ is the standard regular unit simplex in $\R^n$. ($A_{\mu}$ is thus not unique; it is defined
up to a permutation of the vertices of $\Delta_n$.) Consider the polar decomposition $A_{\mu}=S_{\mu}R_{\mu}$, where $S_{\mu}$ and $R_{\mu}$ are, respectively, the symmetric part and the orthogonal part.
The identities $S_{\mu} = \sqrt{A_{\mu}A^t_{\mu}}$ and $R_{\mu} = A_{\mu} S^{-1}_{\mu}$ show that the polar decomposition depends continuously on the non-degenerate linear map $A_{\mu}$.

\begin{definition}
Let $\mu$ be a nice measure on $\R^n$ with insufficient depth.
Using already introduced notation set:
\[
 \Delta(\mu): = R_{\mu}(\Delta_n).
\]
\end{definition}
Since $R_{\mu}$ is an orthogonal operator, the simplex $\Delta(\mu)$ is regular and unit. Moreover, the simplex $\Delta(\mu)$ does not depend on the choice of the simplex $\Delta_n$, because if we choose another $T (\Delta_n)$, with $T$ orthogonal, instead of $\Delta_n$ then $A_{\mu}(\Delta_n) = A' (T (\Delta_n))$ implies that the polar decompositions are $A_{\mu} = S_{\mu}R_{\mu}$ and $A' = S_{\mu}R_{\mu}T^{-1}$, so $R_{\mu}T^{-1}T(\Delta_n)= R_{\mu}(\Delta_n)$.
Hence $\Delta(\mu)$ is indeed only a function of $\mu$.
Finally, since $R_{\mu}$ depends continuously on $\mu$, the same is true for the simplex $\Delta(\mu)$.

\medskip
If we restrict ourselves to nice measures, then Theorem~\ref{Th : stronger} below is a stronger statement than our main result Theorem \ref{Th : Main}.
Theorem~\ref{Th : stronger} will be proved in Section \ref{sec : proof}.

\begin{theorem}\label{Th : stronger}
Let $m \geq 1$, $n \geq 2$ be integers, and let
\begin{compactitem}
	\item $N \ge 2m+n-1$ if $n+1$ is not a power of $2$, and
	\item $N \ge 3m+n-1$ if $n+1$ is a power of $2$.
\end{compactitem}
For every collection of $m$ nice measures $\mu_1, \dots, \mu_m$ on $\R^N$, there exists an $n$-dimensional linear subspace $\Gamma$ with the property that all marginal measures $\Gamma_*\mu_1,\dots,\Gamma_*\mu_m$ have sufficient depth, that is,
\[
\depth(\Gamma_*\mu_1) \geq \frac{1}{n + 1} + \frac{1}{3(n + 1)^3},\ \dots, \ \depth(\Gamma_*\mu_m) \geq \frac{1}{n + 1} + \frac{1}{3(n + 1)^3},
\]
and in addition
\[
\cp(\Gamma_* \mu_1)=\dots=\cp(\Gamma_* \mu_m).
\]
\end{theorem}

Indeed, by taking $x = \cp(\Gamma_* \mu_1)$ we immediately prove Theorem~\ref{Th : Main} for nice measures. As mentioned before,
the case of nice measures implies Theorem~\ref{Th : Main} in general.

\section{Sections of the space of centered regular simplices}

The problem in question makes us study the space of all centered regular simplices of fixed size contained in the fibers of a real vector bundle.
Let $\xi$ be the $n$-dimensional real vector bundle $F\longrightarrow E\longrightarrow B$ endowed with a Euclidean metric on fibers.
Now we want to identify the space of all centered regular simplices of fixed size contained in the fibers of $\xi$.
Let us pass to the universal situation, consider $\xi=f^*\gamma^n$ as the pullback of the tautological $n$-dimensional real vector bundle  $\gamma^n$ along a classifying map $f \colon B\longrightarrow \BO(n)$.
The classifying space $\BO(n)$ can be identified with the infinite Grassmann manifold $G_{n}(\R^{\infty})$, and a model for $\EO(n)$ can be taken to be the Stiefel manifold $V_{n}(\R^{\infty})$ of all orthonormal $n$-frames in $\R^{\infty}$.

The space of all centered regular simplices of fixed size contained in the fibers of the tautological bundle $\gamma_n$ with ordered vertices is parameterized by the Stiefel manifold $V_{n}(\R^{\infty})$.
Indeed, we can identify $V_{n}(\R^{\infty})$ with the space of all isometries $\R^n \longrightarrow \mathbb R^{\infty}$.
Now fix a unit centered regular $n$-simplex $\Delta_n\subset \R^n$ and make the correspondence $\varphi\longmapsto \varphi(\Delta_n)$ where $\varphi\in V_{n}(\R^{\infty})$.
This identifies $V_{n}(\R^{\infty})$ with the space of all centered regular simplices of fixed size contained in the fibers of  $\gamma^n$.

If the orientation on the vertices is dropped then the resulting space is parameterized by $V_{n}(\R^{\infty})/\Sym_{n+1}$, where the symmetric group $\Sym_{n+1}$ action is fiberwise, and on a fiber it is given by the permutation of vertices on the fixed regular simplex in $\R^n$ (a model of the fiber).
The action of $\Sym_{n+1}$ on the Stiefel manifold $V_{n}(\R^{\infty})$ is extended from the action on $\mathbb R^n$ in the standard way.
In particular, the quotient space $V_{n}(\R^{\infty})/\Sym_{n+1}$ is also a model for $\B\Sym_{n+1}$.

Thus, the space of all centered regular simplices of fixed size contained in the fibers of $\gamma_n$ with unordered vertices is, a model for, $\B\Sym_{n+1}$ and can also be seen as a total space of the fiber bundle $\eta \ : \ \OO(n)/\Sym_{n+1}\longrightarrow  \B\Sym_{n+1} \longrightarrow \BO(n)$.
The projection map $\sigma_n \colon \B\Sym_{n+1} \longrightarrow \BO(n)$ is induced by the representation $\rho\colon\Sym_{n+1}\longrightarrow O(n)$ obtained from the permutation of the vertices of the regular simplex.
This representation can also be described as the $\Sym_{n+1}$-representation $W_{n+1}:=\{(x_1,\dots,x_{n+1})\in\R^{n+1} : \sum x_i=0\}$ where the action of $\Sym_{n+1}$ is given by coordinate permutation.
Such a representation is typically called the reduced regular representation.

\medskip
Now we study the (non-)existence of a continuous selection of the centered regular simplices of fixed size contained in the fibers of $\xi$.
Thus we want to answer a question of the non-)existence of a section of the pullback bundle 
\[
\xymatrix{
\OO(n)/\Sym_{n+1}\ar[r] & E(f^*\eta)\ar[r] & B.
}
\]
Here $E(f^*\eta)$ denotes the total space of the pullback bundle $f^*\eta$.
{\em What can obstruct the existence of such a section?}

Let us consider the following pullback square and its associated diagram in cohomology:
\[
\xymatrix{
E(f^*\eta)\ar[r]^-{\hat{f}}\ar[d]^{\sigma} &
\B\Sym_{n+1}\ar[d]^{\sigma_n}   &
                                &
H^*(E(f^*\eta);\F_2)            &
H^*(\B\Sym_{n+1};\F_2)\ar[l]_-{\hat{f}^*}  \\
B\ar[r]^-{f} 				    &
\BO(n)  						&
					            &
H^*(B;\F_2)\ar[u]_{\sigma^*}      &
H^*(\BO(n);\F_2)\ar[l]_-{f^*}\ar[u]_{\sigma_n^*}.
}
\]
We claim that: If there exists a cohomology class $\oo\in H^*(\BO(n); \F_2)$ with the property that
\[
\sigma_n^* (\oo) = 0
\qquad\text{and}\qquad
f^*(\oo)\neq 0,
\]
then there cannot be any continuous section $s\colon B\longrightarrow E(f^*\eta)$ of the bundle $f^*\eta$.
Indeed, let us assume the opposite, that there exists a continuous section $s\colon B\longrightarrow E(f^*\eta)$.
Since $\sigma\circ s=\mathrm{id}_{B}$ and consequently $s^*\circ \sigma^*=\mathrm{id}_{H^*(B;\F_2)}$ we reached the following contradiction
\[
0\neq f^*(\oo) = (s^*\circ \sigma^*)(f^*(\oo))= (s^*\circ \sigma^*\circ f^*)(\oo)=
 (s^*\circ (\hat{f}^* \circ \sigma_n^*))(\oo)=0.
\]
Thus, assuming the notions already introduced, we have obtained the following criterion.

\begin{lemma}
\label{lemma:charclass}
If there exists a cohomology class $\oo\in H^*(\BO(n); \F_2)$ with the property that
\[
\sigma_n^* (\oo) = 0
\qquad\text{and}\qquad
f^*(\oo)\neq 0,
\]
then there cannot exist a section of the fiber bundle $f^*\eta$ of all centered regular simplices of fixed size contained in the fibers of the vector bundle $\xi$.
\end{lemma}

\medskip
Now consider a more general situation.
Starting with the vector bundle $\xi$ we want to cover the base space $B=U_1\cup\dots\cup U_m$ by $m$ open sets and have for each set $U_i$ a continuous (partial) section $s_i$ over $U_i$ of the fiber bundle $f^*\eta$ of centered regular simplices of fixed size contained in the fibers of $\xi$.
{\em What can be used for obstruction the existence of such a covering?}

\begin{lemma}
\label{lemma:cover}
Let $\oo\in H^*(\BO(n); \mathbb F_2)$ be a cohomology class with the property that $\sigma_n^* (\oo) = 0$ and $f^*(\oo)^m \neq 0$.
Then there cannot exist $m$ continuous (partial) sections of the fiber bundle $f^*\eta$ over $m$ open sets $U_1,\dots,U_m$ covering $B$.
\end{lemma}
\begin{proof}
We use the property of cohomology multiplication in a way that goes back to classical work of Lusternik, Schnirelmann, and Schwarz see for example \cite[Thm.\,4]{Schwartz}.
From the assumption that $f^*(\oo)^m\neq 0$ we have that $f^*(\oo)\neq 0$ when restricted to some of $U_i$.
Now Lemma~\ref{lemma:charclass} applied to this $U_i$ and the cohomology class $\oo$ yields the proof of lemma.
\end{proof}

In the next step we identify suitable cohomology classes that can be used as obstructions for the existence of partial sections of the bundle $f^*\eta$.

\begin{lemma}
\label{lemma:powtwo}
If $n+1\geq 4$ is a power of $2$ then $\sigma_n^*(w_1(\gamma^n)w_n(\gamma^n))=0$.
\end{lemma}

\begin{proof}
Let $n+1=2^k\geq 4$, and let $w_i:= \sigma_n^* (w_i(\gamma^n))$ for all $i\geq 0$.
As before, $\rho \colon \Sym_{n+1} \longrightarrow \OO(n)$ denotes the representation given by permuting the vertices of a regular simplex in $\R^n$, and $\sigma_n\colon \B\Sym_{n+1} \longrightarrow \BO(n)$ is the corresponding map of classifying spaces.
The cohomology classes $w_i$ are called the Stiefel--Whitney classes of the representation $\rho$, since the vector bundle $\sigma_n^*\gamma^n$ has a description
\[
\xymatrix{
 W_{n+1}\ar[r]&\E\Sym_{n+1}\times_{\Sym_{n+1}} W_{n+1}\ar[r] & \E\Sym_{n+1}/\Sym_{n+1} = \B\Sym_{n+1},
}
\]
where the action of $\Sym_{n+1}$ on $ W_{n+1}=\{(x_1,\dots,x_{n+1})\in\R^{n+1} : \sum x_i=0 \}$ given by $\rho$ is just the coordinate permutation.

Furthermore, let $E_k\cong (\Z/2)^k$ be the elementary abelian subgroup of $\Sym_{2^k}$ given by the regular embedding $(\mathrm{reg})\colon E_k\longrightarrow \Sym_{2^k}$,~\cite[Ex.\,III.2.7]{Adem-Milgram}.
The regular embedding is given by the left translation action of $E_k$ on itself.
According to \cite[Cor.\,VI.1.4]{Adem-Milgram} the subgroups $\Sym_{2^{k-1}}\times\Sym_{2^{k-1}}$ and $E_k$ detect the cohomology $H^*(\B\Sym_{2^k}; \F_2)$, that is, the homomorphism
\[
\mathrm{res}^{\Sym_{2^k}}_{\Sym_{2^{k-1}}\times\Sym_{2^{k-1}}}\oplus \mathrm{res}^{\Sym_{2^k}}_{E_k}\colon
H^*(\B\Sym_{2^k}; \F_2)\longrightarrow H^*(\B(\Sym_{2^{k-1}}\times\Sym_{2^{k-1}}); \F_2)\oplus H^*(\B E_k; \F_2)
\]
is an injection.
The image of the restriction of $H^*(\Sym_{2^k};\mathbb F_2)$ from $\Sym_{2^k}$ to $E_k$ is the ring of $\mathrm{GL}_k(\F_2)$-invariants $H^*(E_k;\F_2)^{\mathrm{GL}_k(\F_2)}$, \cite[Ex.\,III.2.7]{Adem-Milgram}. There are specific elements
\[
d_{k,s}\in H^{2^k-2^s}(E_k;\F_2)^{\mathrm{GL}_k(\F_2)}
\]
for $0\leq s\leq k-1$, called the Dickson invariants, such that $H^*(E_k;\F_2)^{\mathrm{GL}_k(\F_2)}$ is isomorphic to $\F_2[d_{k,k-1},\dots,d_{k,0}]$ as a graded $\F_2$-algebra, \cite[Thm.\,III.2.4]{Adem-Milgram}.
From~\cite[Lemma\,3.26,\,p.\,59]{madsen-milgram} we have that $\mathrm{res}^{\Sym_{2^k}}_{E_k}(w_{2^k-2^s})=d_{k,s}$ for $0\leq s\leq k-1$.
Thus, $w_1\in \mathrm{ker}\big(\mathrm{res}^{\Sym_{2^k}}_{E_k}\big)$.
On the other hand, from \cite[Cor.\,3.30,\,p.\,61]{madsen-milgram}, the kernel of the second restriction is
\[
\mathrm{ker}\Big(\mathrm{res}^{\Sym_{2^k}}_{\Sym_{2^{k-1}}\times\Sym_{2^{k-1}}}\Big)=w_{2^k-1}\cdot \mathbb F_2[w_{2^k-1},\dots,w_{2^k-2^{k-1}}].
\]
In particular we have that $w_{2^k-1}\in \mathrm{ker}\Big(\mathrm{res}^{\Sym_{2^k}}_{\Sym_{2^{k-1}}\times\Sym_{2^{k-1}}}\Big)$.
Combining these facts, we have that
$
w_1w_{2^k-1}\in \mathrm{ker}\Big(\mathrm{res}^{\Sym_{2^k}}_{\Sym_{2^{k-1}}\times\Sym_{2^{k-1}}}\oplus \mathrm{res}^{\Sym_{2^k}}_{E_k}\Big)
$, and consequently $w_1w_{2^k-1}=0$.
This concludes the proof of the lemma.

\smallskip
Assuming the ``detection'' property, a more direct explanation of the fact $w_1w_{2^k-1}=0$ can be given as follows. After restricting to $E_k$ the representation $\rho$ becomes orientable (we use $k\ge 2$), hence the class $w_1$ vanishes after restriction to $E_k$.
After restricting to $\Sym_{2^{k-1}}\times\Sym_{2^{k-1}}$ the representation $\rho$ gets a non-zero invariant vector, for example $(1,\dots, 1,  -1,\dots,-1)$ where $1$ and $-1$ appear the same number of times.
Consequently, the corresponding vector bundle will have a trivial summand implying that the top Stiefel--Whitney class $w_n$ vanishes under the restriction to $\Sym_{2^{k-1}}\times\Sym_{2^{k-1}}$.
Hence, again $w_1w_{2^k-1}=0$.
\end{proof}

\noindent
Notice that in the proof of the previous lemma we could have identified all the monomials in the Stiefel--Whitney classes of the bundle $\gamma_n$ that vanish in $H^*(\B\Sym_{2^k};\mathbb F_2)$.

\begin{lemma}
\label{lemma:ozaydin}
If $n+1$ is not a power of $2$ then $\sigma_n^*(w_n(\gamma^n))=0$.
\end{lemma}
\begin{proof}
Now $n+1$ is not a power of $2$.
As in the proof of the previous lemma set $w_i:= \sigma_n^* (w_i(\gamma^n))$ for all $i\geq 0$.
Again $\rho \colon \Sym_{n+1} \longrightarrow \OO(n)$ denotes the representation given by permuting the vertices of a regular simplex in $\R^n$, and $\sigma_n\colon \B\Sym_{n+1} \longrightarrow \BO(n)$ is the corresponding map of classifying spaces.
The cohomology classes $w_i$ are  Stiefel--Whitney classes of  the vector bundle $\sigma_n^*\gamma^n$:
\[
\xymatrix{
 W_{n+1}\ar[r]&\E\Sym_{n+1}\times_{\Sym_{n+1}} W_{n+1}\ar[r] & \E\Sym_{n+1}/\Sym_{n+1} = \B\Sym_{n+1}.
}
\]
The Stiefel--Whitney classes $w_n$ live in $H^n(\Sym_{n+1};\F_2)$.

Now consider an arbitrary Sylow $2$-subgroup $\Sym_{n+1}^{(2)}$ of the symmetric group $\Sym_{n+1}$.
The inclusion map $\alpha_n\colon \Sym_{n+1}^{(2)}\longrightarrow \Sym_{n+1}$ induces the map $\beta_n\colon\B \Sym_{n+1}^{(2)}\longrightarrow \B\Sym_{n+1}$ between the classifying spaces that in turn induces the restriction map 
\[
\mathrm{res}^{\Sym_{n+1}}_{\Sym_{n+1}^{(2)}}\colon H^*(\Sym_{n+1};\F_2)\longrightarrow H^n(\Sym_{n+1}^{(2)};\F_2).
\] 
The restriction map $\mathrm{res}^{\Sym_{n+1}}_{\Sym_{n+1}^{(2)}}$ is injective, see \cite[Prop.\,III.9.5(ii) and Thm.\,III.10.3]{Brown}.
Thus it suffices to prove that  
\[
\mathrm{res}^{\Sym_{n+1}}_{\Sym_{n+1}^{(2)}}(\sigma_n^*(w_n(\gamma^n)))=
\mathrm{res}^{\Sym_{n+1}}_{\Sym_{n+1}^{(2)}}(w_n(\sigma_n^*\gamma^n))=
w_n(\beta_n^*\sigma_n^*\gamma^n)=
0.
\]
Since $n+1$ is not a power of $2$ the Sylow $2$-subgroup $\Sym_{n+1}^{(2)}$ does not act on the set $\{1,\dots,n+1\}$ transitively. 
Thus the fixed point set $W_{n+1}^{\Sym_{n+1}^{(2)}}$ with respect to the action of the Sylow $2$-subgroup $\Sym_{n+1}^{(2)}$ on $W_{n+1}$ is a vector space of positive dimension.
Consequently, the pull-back vector bundle $\beta_n^*\sigma_n^*\gamma^n$ can be decomposed into a Whitney sum of two vector bundles of positive dimensions where one of them is a trivial one. 
Hence, the top Stiefel--Whitney class $w_n(\beta_n^*\sigma_n^*\gamma^n)$ of the pull-back vector bundle $\beta_n^*\sigma_n^*\gamma^n$ has to vanish implying that $\sigma_n^*(w_n(\gamma^n))=0$.
\end{proof}

\section{Proof of Theorem \ref{Th : Main} and Theorem \ref{Th : stronger}}
\label{sec : proof}

We have shown in Section 2 that Theorem \ref{Th : Main} follows from Theorem~\ref{Th : stronger}. 
In turn, we  prove Theorem~\ref{Th : stronger} by contradiction.
If a collection of nice measures $\mu_1, \mu_2, \dots, \mu_m$ provides a counterexample to Theorem~\ref{Th : stronger}, then for any $n$-dimensional linear subspace $\Gamma\in G_{n}(\R^N)$ the following holds:
\begin{compactitem} 
\item the points $\cp(\Gamma_*\mu_1),\dots,\cp(\Gamma_*\mu_m)$ associated to the measures $\Gamma_*\mu_1,\dots,\Gamma_*\mu_m$ do not all coincide, {\bf or}	
\item at least one of the measures, for example $\Gamma_*\mu_i$, induces a regular simplex $\Delta(\Gamma_*\mu_i)$ contained in the flat $\Gamma$ centered at its origin in a continuous way.
\end{compactitem}
In other words, the Grassmannian $G_{n}(\R^N)$ can be
covered by $m+1$ open subsets $U_0, U_1,\dots, U_m$ satisfying the conditions:
\begin{compactenum} 
\item The restriction of the Whitney power $(\gamma^{n}_{N})^{\oplus (m-1)}$ to $U_0$ has a continuous nonzero section.
\item The restriction of the bundle $i^*\eta$ to each $U_i$, $1 \leq i \leq m$, has a continuous section.
	  Here $i\colon G_n(\R^N)\longrightarrow G_N(\R^{\infty})$ denotes the natural inclusion.	
\end{compactenum}
Indeed, let $U_0$ be the subset of $G_{n}(\R^N)$ where the points $\cp(\Gamma_*\mu_1),\dots,\cp(\Gamma_*\mu_m)$ do not all coincide.
Then the $(m - 1)$-tuple of vectors
\[
\big(\cp(\Gamma_* \mu_{i + 1}) - \cp(\Gamma_* \mu_i)\big)_{i = 1}^{m - 1}
\]
gives a section required in (1).
Next, let $U_i$ be the open subset of all $\Gamma \in G_{n}(\R^N)$ such that $\depth(\Gamma_*\mu_i) \geq \frac{1}{n + 1} + \frac{1}{3(n + 1)^3}$. Then the related measure $\Gamma_*\mu_i$ defines a regular simplex
$\Delta(\Gamma_*\mu_i)$ creating a section required in (2).
Finally, since $\mu_1, \mu_2, \dots, \mu_m$ is a counterexample to Theorem~\ref{Th : stronger}, the collection $U_0, U_1,\dots, U_m$ is an open cover of $G_{n}(\R^N)$.

\medskip
We will show that such a covering cannot exist by proving that appropriate characteristic classes do not vanish.

For the rest of the proof we use \v{C}ech cohomology for its continuity property: The cohomology of a closed set is a limit of cohomology of its open neighborhoods.

\subsection{}
Let $n+1$ be not a power of $2$, and let $m \geq 1$ be an integer.
Without loss of generality we can assume that $N = 2m+n-1$.
Consider the subset of the  Grassmannian $G_{n}(\R^N)$ where all points $\cp(\Gamma_*\mu_1),\dots,\cp(\Gamma_*\mu_m)$ coincide, that is
\[
X=\{\Gamma\in G_{n}(\R^N) : \cp(\Gamma_*\mu_1)=\dots=\cp(\Gamma_*\mu_m)\}.
\]
From \cite[Lemma\,1.2]{hil1980A} we have that the cohomology class $w_n(\gamma^n_N)^{N-n}$ does not vanish.
Consequently, the cohomology class $w_n(\gamma^n_N)^{N-n-m+1}$ does not vanish along the restriction map $H^*(G_{n}(\R^N);\F_2)\longrightarrow H^*(X;\F_2)$.

Now, according to our assumption, the subspace $X$ is covered by the open sets $U_1,\dots, U_m$ that allow sections of the bundle $i^*\eta$ over each $U_i$.
On the other hand, since $n+1$ is not a power of $2$, using Lemma \ref{lemma:cover} and Lemma \ref{lemma:ozaydin}, we have that $X$ cannot be covered by $N-n-m+1=m$ such open subsets. We have reached a contradiction.
Thus, such a covering does not exist and the proof of the theorem in the case $n+1$ is a power of $2$ is complete.

\subsection{}
Assume $n+1$ is not a power of $2$ and let $m \geq 1$ be an integer.
Without loss of generality we assume that $N = 3m+n-1$.
Let us consider the cohomology class  $w_1(\gamma^n_N)^{m}w_n(\gamma^n_N)^{2m-1}$.
The cohomology class $w_n(\gamma^n_N)^{m-1}$ vanishes along the restriction $H^*(G_n(\R^N);\F_2)\longrightarrow H^*(U_0;\F_2)$, and according to Lemma \ref{lemma:charclass} the cohomology class $w_1(\gamma^n_N)w_n(\gamma^n_N)$  vanishes along any of the restrictions $H^*(G_n(\R^N);\F_2)\longrightarrow H^*(U_i;\F_2)$ for $1\leq i\leq m$.
Since $U_0,U_1,\dots,U_m$ is a covering of  the Grassmannian $G_n(\R^N)$
the product class $w_1(\gamma^n_N)^{m}w_n(\gamma^n_N)^{2m-1}=0$ vanishes over $G_n(\R^N)$.
We reach a contradiction with the assumption of the existence of a cover by proving that $w_1(\gamma^n_N)^{m}w_n(\gamma^n_N)^{2m-1}\neq 0$.
This will conclude the proof of the theorem.
For this fact we offer two different proofs.

\subsubsection{}
 Let us first argue in terms of intersections of the Poincar\'e duals to the classes  $w_n(\gamma^n_N)^{2m-1}$ and $w_1(\gamma^n_N)^{m}$ in the homology modulo two.
A Poincar\'e dual of the class $w_n(\gamma^n_N)^{2m-1}$ is presented by the space $X$ where $2m-1$ generic sections of $\gamma^n_N$ intersect.
Following the framework of the proof of \cite[Lemma\,8]{Karasev2007} we choose $2m-1$ last basis vectors $e_{N-2m+2}, \dots, e_N$ of $\R^N$ and project them onto a subspace $\Gamma\in G_{n}(\R^N)$.
They produce $2m-1$ sections of the vector bundle $\gamma^n_N$.
The space where all the sections are zero is the naturally embedded Grassmannian $G_n(\R^{N-2m+1})\subseteq G_n(\R^N)$, and it is easy to check that the intersection is non-transversal.

Next we consider a piece-wise smooth modulo $2$ cycle $Y$ that represents the Poincar\'e dual of the class $w_1(\gamma^n_N)^{m}$ in $G_n(\R^N)$ that is transversal to the submanifold $G_n(\R^{N-2m+1})$. We have to show that the intersection $Y\cap G_n(\R^{N-2m+1})$ represents a non-zero homology class in $H_*(G_{n}(\R^N); \F_2)$.
It is sufficient to prove that it represents a non-zero homology class in $H_*(G_n(\R^{N-2m+1});\F_2)$.
Indeed, since the Schubert cell decomposition of $G_n(\R^{N-2m+1})$ is a part of the Schubert cell decomposition of $G_n(\R^{N})$ with every cell representing an independent generator in mod $2$, implies that the natural homology map
\[
H_*(G_n(\R^{N-2m+1}); \mathbb F_2) \longrightarrow H_*(G_n(\R^{N}); \mathbb F_2)
\]
is an injection.

Applying the Poincar\'e duality once again we need to prove that the class $w_1(\gamma^n_{N-2m+1})^m$ is non-zero in $H^*(G_n(\R^{N-2m+1});\F_2)$.
A simple sufficient condition for this (see~\cite[Thm.\,3.4]{hil1980A} or \cite[Sec.\,2]{kardol2011}) is the inequality
\[
N-2m- n+1 \geq  m \quad\Longleftrightarrow\quad N= 3m+n-1.
\]
Thus,  $w_1(\gamma^n_N)^{m}w_n(\gamma^n_N)^{2m-1}\neq 0$.

\subsubsection{}
For the second proof we use Pieri's formula and presentation of Stiefel--Whitney classes of the tautological bundle $\gamma^n_{N}$ over $G_n(\R^{N})$  in the form of Schubert cocycles $(a_1,\dots,a_n)$ where $1\leq a_1\leq\dots\leq a_n\leq N-n$.
Following the presentation in \cite{hil1980A} we have that
\begin{compactitem}
\item $w_i(\gamma^n_N)=(0,\dots, 0, 1,\dots,1)$ where $1$ occurs $i$ times for $1\leq i \leq n$,
\item $\bar{w}_j(\gamma^n_N)=(0,\dots, 0, j)$ for $1\leq j\leq N-n$,
\item (\cite[Lemma\,1.2]{hil1980A}) $w_n(\gamma^n_N)^k=(k,\dots,k)$,
\item (Pieri's formula) $(a_1,\dots,a_n)\,\bar{w}_j(\gamma^n_N)=\sum (b_1,\dots,b_n)$ where the sum is over all $(b_1,\dots,b_n)$ with the property that
	\begin{compactitem}
	\item $a_i\leq b_i\leq a_{i+1}$ for all $1\leq i\leq n$, where
	\item $a_{n+1}=N-n$, and
	\item $b_1+\dots+b_n=j+a_1+\dots+a_n$.
	\end{compactitem}	
\end{compactitem}
Using these facts we compute:
\begin{eqnarray*}
	w_1(\gamma^n_N)^{m}w_n(\gamma^n_N)^{2m-1} & = & (2m-1,\dots,2m-1,2m-1)\,w_1(\gamma^n_N)^{m}\\
& = & (2m-1,\dots,2m-1,2m-1)\,\bar{w}_1(\gamma^n_N)^{m}\\
& = & (2m-1,\dots,2m-1,2m)\,\bar{w}_1(\gamma^n_N)^{m-1}\\
& = & \big((2m-1,\dots,2m,2m)+ (2m-1,\dots,2m-1,2m+1)\big)\,\bar{w}_1(\gamma^n_N)^{m-2}\\
& = & \dots \\
& = & A + (2m-1,\dots,2m-1,3m-1).
\end{eqnarray*}
where $A$ is a sum of some Schubert cocycles different from the cocycle $(2m-1,\dots,2m-1,3m-1)$.
Since $3m-1=N-n$ the cocycle $(2m-1,\dots,2m-1,3m-1)$ is not zero and consequently $w_1(\gamma^n_N)^{m}w_n(\gamma^n_N)^{2m-1}\neq 0$.

\medskip
Thus we have concluded the proof of Theorem \ref{Th : Main} and Theorem \ref{Th : stronger}.

\section{Concluding remarks}

The last step in the proof of Theorem \ref{Th : Main} can be further improved using the following results of Hiller~\cite{hil1980A} \cite{hil1980B}.

\begin{lemma}
\label{orgindex}
Let $2n\le N$, otherwise replace $n$ by $N-n$ in practical applications.
Let $2^s$ be the minimal power of two, satisfying $2^s\ge N$.
Then
\begin{compactenum}[\rm \quad(1) ]
\item if $n = 1$ then $w_1(\gamma^n_N)^{N-1} \neq 0$,
\item if $n = 2$ then $w_1(\gamma^n_N)^{2^s-2} \neq 0$,
\item if $n > 2$ then in the case $N=2n=2^s$ we have $w_1(\gamma^n_N)^{2^{s-1}} \neq 0$ and $w_1(\gamma^n_N)^{2^s-2} \neq 0$ in other cases.
\end{compactenum}
In all cases $w_1(\gamma^n_N)^{N-n}\neq 0$ and this cannot be improved for $n=1$, $n=2$ and $N=2^s$.
\end{lemma}

\noindent
This improvement is not a content of our main result since it would require a complicated statement.

\medskip
Finally we discuss the relationship of our main result with the result of Magazinov and P\'or \cite{MP16}.
A careful reader might have noticed that Theorem~\ref{Th : Main} does not contain Theorem~\ref{Theorem:Centerline} as a particular case in the case when $n+1$ is a power of $2$.
Indeed, when $n+1$ is a power of $2$ the proof presented in Section \ref{sec : proof} cannot be used because the class $w_1(\gamma_{n+1}^n)w_n(\gamma_{n+1}^n)$ vanishes for dimension reasons, that is  $G_n(\R^{n+1})\cong \mathbb R \mathrm P^n$ has dimension $n$.
Nevertheless, this case can be dealt with by a modified argument that follows.

Assume that a section of the regular unit simplex bundle $i^*\eta$ over $G_n(\mathbb R^{n+1})$ is given.
Recall that we consider simplices with unordered  vertices.
Over each $\Gamma\in G_n(\mathbb R^{n+1})$ we have $n+1$ vertices of the simplex in the fiber of $\gamma_{n+1}^n$; in total the set of vertices of all simplices in all fibers produces an $(n+1)$-sheet covering $C\longrightarrow G_n(\mathbb R^{n+1})$.
Since $\pi_1(G_n(\mathbb R^{n+1}))\cong\Z/2$ the covering space $C$ must split into connected components
\[
C = C_1\cup\dots\cup C_m
\]
such that every projection $C_i\longrightarrow G_n(\mathbb R^{n+1})$ is a covering with either one or two sheets.

If some $C_i\longrightarrow G_n(\mathbb R^{n+1})$ is a one sheet covering then it just means that the corresponding vertex produces a section of the canonical bundle $\gamma_{n+1}^n$.
This cannot be because $w_n(\gamma_{n+1}^n)\neq 0$.

If, on the other hand, some $C_i\longrightarrow G_n(\mathbb R^{n+1})$ has two sheets then we have a continuous selection of a pair of vertices $v_1,v_2$ from the regular simplex in every fiber.
Of course, the pair $\{v_1,v_2\}$ is defined up to the order.
But then, in case $n\ge 2$, we may take $v_1+v_2$ as a continuous nonzero section of $\gamma_{n+1}^n$ and obtain a contradiction with $w_n(\gamma_{n+1}^n)\neq 0$ as well.

Thus a section of the regular unit simplex bundle $i^*\eta$ over $G_n(\mathbb R^{n+1})$ cannot exist.
This concludes the argument.

\small
\bibliography{references}{}

\providecommand{\bysame}{\leavevmode\hbox to3em{\hrulefill}\thinspace}
\providecommand{\MR}{\relax\ifhmode\unskip\space\fi MR }
\providecommand{\MRhref}[2]{%
  \href{http://www.ams.org/mathscinet-getitem?mr=#1}{#2}
}
\providecommand{\href}[2]{#2}
\begin{thebibliography}{10}

\bibitem{Adem-Milgram}
Alejandro Adem and James~R. Milgram, \emph{Cohomology of {F}inite {G}roups},
  second ed., Grundlehren der Mathematischen Wissenschaften, vol. 309,
  Springer-Verlag, Berlin, 2004.

\bibitem{Brown}
Kenneth~S. Brown, \emph{Cohomology of groups}, Graduate Texts in Mathematics,
  vol.~87, Springer-Verlag, New York, 1994, Corrected reprint of the 1982
  original.

\bibitem{BMN10}
Boris Bukh, Ji{\v{r}}{\'{\i}} Matou{\v{s}}ek, and Gabriel Nivasch,
  \emph{Stabbing simplices by points and flats}, Discrete Comput. Geom.
  \textbf{43} (2010), no.~2, 321--338.

\bibitem{Dol87}
Vladimir~L. Dol'nikov, \emph{Generalized transversals to families of sets in
  $\mathbb{R}^n$ and connections between the theorems of {H}elly and {B}orsuk},
  Dokl. Akad. Nauk SSSR (N.S.) \textbf{297} (1987), no.~4, 787--780, (In
  Russian).

\bibitem{Dol92}
\bysame, \emph{A generalization of the {H}am {S}andwich theorem}, Mathematical
  Notes \textbf{52} (1992), no.~2, 771--779.

\bibitem{kardol2011}
Vladimir~L. Dol'nikov and Roman~N. Karasev, \emph{Dvoretzky type theorems for
  multivariate polynomials and sections of convex bodies}, Geom. Funct. Anal.
  \textbf{21} (2011), no.~2, 301--318.

\bibitem{hil1980A}
Howard~L. Hiller, \emph{On the cohomology of real {G}rassmanians}, Trans. Amer.
  Math. Soc. \textbf{257} (1980), no.~2, 521--533.

\bibitem{hil1980B}
\bysame, \emph{On the height of the first {S}tiefel-{W}hitney class}, Proc.
  Amer. Math. Soc. \textbf{79} (1980), no.~3, 495--498.

\bibitem{Karasev2007}
Roman~N. Karasev, \emph{Tverberg's transversal conjecture and analogues of
  nonembeddability theorems for transversals}, Discrete Comput. Geom.
  \textbf{38} (2007), no.~3, 513--525.

\bibitem{Kl10}
Bo'az Klartag, \emph{On nearly radial marginals of high-dimensional probability
  measures}, J. Eur. Math. Soc. \textbf{12} (2010), no.~3, 723--754.

\bibitem{madsen-milgram}
Ib~Madsen and James~R. Milgram, \emph{The classifying spaces for surgery and
  cobordism of manifolds}, Annals of Mathematics Studies, vol.~92, Princeton
  University Press, Princeton, N.J.; University of Tokyo Press, Tokyo, 1979.

\bibitem{Magazionov2016}
Alexander Magazinov, \emph{A \v {Z}ivaljevi\' c--{V}re\' cica--{D}olnikov-type
  theorem for {S}uper-{R}ado depth}, Preprint, 8 pages, Jun 2016;
  \href{http://arxiv.org/abs/1606.08225v1}{arXiv:1606.08225v1}.

\bibitem{MP16}
Alexander Magazinov and Attila P{\'o}r, \emph{An improvement on the {R}ado
  {B}ound for the {C}enterline {D}epth}, Discrete Comput. Geom. \textbf{59}
  (2018), no.~2, 477--505.

\bibitem{Ra46}
R.~Rado, \emph{A theorem on general measure}, J. London Math. Soc. \textbf{21}
  (1946), 291--300 (1947).

\bibitem{Schwartz}
A.~S. {\v{S}}varc, \emph{The genus of a fiber space}, Dokl. Akad. Nauk SSSR
  (N.S.) \textbf{119} (1958), 219--222.

\bibitem{Tao2011}
Terence Tao, \emph{An {I}ntroduction to {M}easure {T}heory}, Graduate Studies
  in Mathematics, vol. 126, American Mathematical Society, Providence, RI,
  2011.

\bibitem{Tuk75}
John~W. Tukey, \emph{Mathematics and the picturing of data}, Proceedings of the
  {I}nternational {C}ongress of {M}athematicians ({V}ancouver, {B}. {C}.,
  1974), {V}ol. 2, Canad. Math. Congress, Montreal, Que., 1975, pp.~523--531.

\bibitem{ZV90}
Rade~T. {\v{Z}}ivaljevi{\'c} and Sini{\v{s}}a~T. Vre{\'c}ica, \emph{An
  {E}xtension of the {H}am {S}andwich {T}heorem}, Bull. London Math. Soc.
  \textbf{22} (1990), no.~2, 183--186.

\end{thebibliography}
\bibliographystyle{amsplain}

\end{document}